\documentclass[12pt,reqno]{amsart}

\usepackage{graphicx}

\usepackage{amssymb}
\usepackage{amsthm}
\usepackage{mathtools}

\theoremstyle{plain}

\newtheorem*{theorem*}{Theorem}
\newtheorem*{conjecture1}{Conjecture 1}
\newtheorem*{conjecture2}{Conjecture 2}

\newtheorem{definition}{Definition}
\newtheorem{theorem}{Theorem}
\newtheorem{lemma}{Lemma}

\usepackage{lineno}

\newcommand{\herm}[2]{h_{#1, #2}}

\title[A new proof of Atiyah's conjecture for $4$ points]{A new proof of Atiyah's conjecture on configurations of four points}

\author{Joseph Malkoun}

\begin{document}
	
\begin{abstract} In Surveys in Differential Geometry, Volume 7, published in 2002 (\cite{Ati-2000}) and Philosophical Transactions of the Royal Society A, Volume 359, published in 2001 (\cite{Ati-2001}), Sir Michael Atiyah introduced what is known as the Atiyah problem on configurations of points, which can be briefly described as the conjecture that the $n$ polynomials (each defined up to a phase factor) associated geometrically to a configuration of $n$ distinct points in $\mathbb{R}^3$ are always linearly independent. The first ``hard'' case is for $n = 4$ points, for which the linear independence conjecture was proved by Eastwood and Norbury in Geometry \& Topology (2), in 2001 (\cite{EasNor2001}).
	
We present a new proof of Atiyah's linear independence conjecture on configurations of four points, i.e. of Eastwood and Norbury's theorem. Our proof consists in showing that the Gram matrix of the $4$ polynomials associated to a configuration of $4$ points in Euclidean $3$-space is always positive definite. It makes use of $2$-spinor calculus and the theory of hermitian positive semidefinite matrices.
\end{abstract}

\maketitle

\section{Introduction}

Inspired by Michael Berry and Jonathan Robbins's approach in \cite{BR1997} to understand the spin-statistics theorem quantum mechanically (and, essentially, geometrically!), Michael Atiyah proposed in a series of related papers in 2000-2001 (\cite{Ati-2000}, \cite{Ati-2001}) a geometric construction which associates smoothly to each configuration of $n$ distinct points in $\mathbb{R}^3$, $n$ complex polynomials of degree at most $n-1$, each defined up to a complex scalar factor only. Atiyah conjectured that any set of $n$ polynomials obtained via this construction is linearly independent over $\mathbb{C}$. 

Atiyah proved linear independence for the case of $n = 3$ points in various ways (cf \cite{Ati-2000} and \cite{Ati-2001} for example), with the $n = 2$ case being trivial. Moreover, in \cite{Ati-2001}, Atiyah  defined a normalized determinant $D$ of the $n$ polynomials in this geometric construction. Numerical calculations performed by Atiyah and Sutcliffe in \cite{Ati-Sut-2002} indicated that $|D| \geq 1$ for any configuration of $n$ distinct points for $n$ up to (at least) $32$. So it was natural for the authors to conjecture that this inequality held for any $n \geq 2$. This was known as Conjecture 2, with Conjecture 1 referring to the linear independence conjecture. Clearly, Conjecture 2 implies Conjecture 1. Atiyah and Sutcliffe also made an even stronger conjecture which they referred to as Conjecture 3 (which implies Conjecture 2), but we shall not discuss it in this article.

Eastwood and Norbury in \cite{EasNor2001} proved the linear independence conjecture (Conjecture 1) for the case of $n = 4$ points. This paper is remarkable because, after expanding $D$ using a computer algebra software (Maple) and obtaining an expression for $D$ involving about $200$ terms, the authors were able to express $D$ in terms of geometric quantities which are obviously nonnegative (triangle inequalities, volumes and so on). They even came close to showing conjecture $2$ for $n = 4$ (they showed that $|D| \geq \frac{15}{16}$).

Building on Eastwood and Norbury's work, Bou Khuzam and Johnson proved in 2014 Conjecture 2 (and even Conjecture 3) for $n = 4$ in \cite{KhuJoh2014} by essentially setting up a linear program and solving it using a computer. At around the same time, Dragutin Svrtan gave his arguments for Conjecture 2 (and Conjecture 3) also for the $n = 4$ case in a talk at the 73-rd ``Seminaire Lotharingien de Combinatoire''. His methods also build on Eastwood and Norbury's formula for the expansion of the normalized determinant for the $n = 4$ case (for which he also provides a human-readable proof).

The methods used to attack the $n = 4$ case rely upon the ``brute force'' expansion of the determinant. One may possibly prove the $n = 5$ case in a similar way after a lot of work (and computers with sufficient processing power), but such methods are obviously limited and do not seem to offer much help with the general case. One may perhaps hope to find a pattern in these formulas for low values of $n$, but so far, such efforts have not proved successful.

It seems clear to the author that one needs to at least supplement Eastwood and Norbury's work by finding a different proof for the $n = 4$ case which avoids the full expansion of Atiyah's determinant function and avoids the use of computer algebra software and the like, before tackling the general case. With this goal in mind, we show in this article that the Gram matrix of the $4$ polynomials associated to a configuration of $4$ distinct points in $\mathbb{R}^3$ is always positive definite, using $2$-spinor calculus and the theory of hermitian positive (semi-)definite matrices. This reproves the linear independence theorem by Eastwood and Norbury for the $n = 4$ case.

\section{Background: the Hopf map and $2$-spinors}

The Hopf map $h: S^3 \to S^2$ is a smooth map from $S^3$ onto $S^2$ with fibers being diffeomorphic to circles. The $3$-sphere $S^3$ can be described as follows.
\[ S^3 = \{ (u, v) \in \mathbb{C}^2 \,; \, |u|^2 + |v|^2 = 1 \}. \]
On the other hand, the $2$-sphere can be described as
\[ S^2 = \{ (\zeta, z) \in \mathbb{C} \times \mathbb{R} \,; \, 
|\zeta|^2 + z^2 = 1 \}.\]
The Hopf map is then defined by
\[ h(u, v) = (2 \bar{u} v, |v|^2 - |u|^2). \]
Note that, if $(\zeta, z) \in S^2$ with $z < 1$, then
\[ h^{-1}(\zeta, z) = \frac{e^{i \theta}}{\sqrt{2(1-z)}}\left(1 - z, \,\zeta\right)\]
where $\theta$ is real. On the other hand,
\[ h^{-1}(0, 1) = e^{i\alpha} (0, 1),\]
where $\alpha \in \mathbb{R}$.

A vector in $\mathbb{C}^2$ may be identified with a linear form on $\mathbb{C}^2$ using the complex symplectic form on $\mathbb{C}^2$ (unique up to a nonzero complex scalar factor), which we will denote by $\omega$. With respect to the standard basis of $\mathbb{C}^2$, $\omega$ gets represented by the following matrix
\begin{equation} \omega = \begin{pmatrix} 0 & 1 \\ -1 & 0 \end{pmatrix}. \label{omega} \end{equation}
Thus for example, a vector $\mathbf{\psi} = (\psi_0, \psi_1)^T \in \mathbb{C}^2$, where $T$ denotes the transpose, gets identified with
\begin{equation} \omega(\mathbf{\psi}, \, \mathbf{w}) = - \psi_1 u + \psi_0 v, \label{identification} \end{equation}
where $\mathbf{w} = (u, v)^T$. So, in particular, if $z < 1$, we have the identification
\[ h^{-1}(\zeta, z) \sim \frac{1}{\sqrt{2(1-z)}}((1-z)v - \zeta u)\]
(up to a phase factor). If we think of $u$ and $v$ as homogeneous coordinates on $\mathbb{P}^1(\mathbb{C})$, then switching to the inhomogeneous coordinate $v/u$, we see that the root of the previous linear form is $\zeta/(1 - z)$, which is nothing but the stereographic projection of the point $(x, y, z) \in S^2 \subset \mathbb{R}^3$, with $x$ and $y$ being the real and imaginary parts of $\zeta$, respectively.

While the original problem was formulated using stereographic projection and the Hopf map, it will be useful to consider a (very slightly) modified Hopf map instead, which is more naturally associated to the Pauli matrices.

We first remark that any hermitian $2 \times 2$ matrix is a linear combination with real coefficients of the $4$ Pauli matrices:
\begin{align*}
	\sigma_0 &= \begin{pmatrix} 1 & 0 \\ 0 & 1 \end{pmatrix} \\
	\sigma_1 &= \begin{pmatrix} 0 & 1 \\ 1 & 0 \end{pmatrix} \\
	\sigma_2 &= \begin{pmatrix} 0 & -i \\ i & 0 \end{pmatrix} \\
	\sigma_3 &= \begin{pmatrix} 1 & 0 \\ 0 & -1 \end{pmatrix}
\end{align*}
Note that $\sigma_0$ is the identity $2 \times 2$ matrix, which we also denote by $\mathbf{1}$.

If $\mathbf{w} = (u, v)^T \in S^3 \subset \mathbb{C}^2$, we first form
\[ \mathbf{w} \mathbf{w}^* = \begin{pmatrix} |u|^2 & u\bar{v} \\
	v \bar{u} & |v|^2 \end{pmatrix} \]
(with the $*$ denoting the conjugate transpose) which is an hermitian $2 \times 2$ matrix with trace equal to $1$. We can then write, in a unique way
\[ \mathbf{w} \mathbf{w}^* = \frac{1}{2}\left(\mathbf{1} - x \sigma_1 - y \sigma_2 - z\sigma_3 \right), \]
where $x$, $y$ and $z$ are real numbers. Noting that $\mathbf{w} \mathbf{w}^*$ has rank $1$, it thus follows that its determinant vanishes. Hence
\[ 0 = \det(\mathbf{w} \mathbf{w}^*) = \frac{1}{4}\left( 1 - x^2 - y^2 - z^2 \right). \]
In other words, $(x, y, z) \in S^2 \subset \mathbb{R}^3$. We now define the modified Hopf map $\tilde{h}: S^3 \to S^2 \subset \mathbb{R}^3$ as follows.
\[ \tilde{h}(u, v) = (x, y, z). \]
The formula for the inverse map, assuming $z < 1$, now reads
\[ \tilde{h}^{-1}(x, y, z) = \frac{e^{i \theta}}{2(1-z)}\left(1-z, -(x+iy)\right)^T \]
and
\[ \tilde{h}^{-1}(0, 0, 1) = e^{i \alpha}(0, 1)^T. \]

Using the modified Hopf map rather than the usual Hopf map does not affect the absolute value of the normalized determinant $D$ mentioned in the previous section (which will be discussed in section \ref{At-det}), so ultimately, it will be equivalent to using the usual Hopf map, as far as the Atiyah problem on configurations is concerned.

\section{A few useful formulas}

If $\mathbf{w}_i \in S^3 \subset \mathbb{C}^2$ with $\mathbf{w}_i = (u_i, v_i)^T$ for $i = 1, 2$, we then denote by $\langle -,\, -\rangle$ the standard hermitian inner product on $\mathbb{C}^2$, i.e.
\begin{equation} \langle \mathbf{w}_1, \, \mathbf{w}_2 \rangle = u_1 \bar{u}_2 + v_1 \bar{v}_2 = \operatorname{tr}(\mathbf{w}_1 \mathbf{w}_2^*) = \mathbf{w}_2^* \mathbf{w}_1. \label{herm} \end{equation}

We therefore deduce that
\begin{align*} & |\langle \mathbf{w}_1 ,\, \mathbf{w}_2 \rangle|^2 \\
= &\, \mathbf{w}_2^* \mathbf{w}_1 \mathbf{w}_1^* \mathbf{w}_2 \\
= &\, \operatorname{tr}(\mathbf{w}_1\mathbf{w}_1^*
\mathbf{w}_2\mathbf{w}_2^*) \\
= &\, \frac{1}{4}\operatorname{tr}\left[ \left( \mathbf{1} - \mathbf{x}_1.\vec{\sigma} \right) \left( \mathbf{1} - \mathbf{x}_2.\vec{\sigma} \right) \right] \\
= &\, \frac{1}{4}\operatorname{tr}\left[ \left(\mathbf{1} - \mathbf{x}_1.\vec{\sigma} - \mathbf{x}_2.\vec{\sigma} + (\mathbf{x}_1, \mathbf{x}_2) \mathbf{1}\right)\right]
\end{align*}
where $\mathbf{x}_i = (x_i, y_i, z_i)^T \in S^2 \subset \mathbb{R}^3$, which is the (modified) Hopf image of $\mathbf{w}_i$, for $i = 1, 2$,  $(-,\, -)$ denotes the standard Euclidean inner product on $\mathbb{R}^3$, $\vec{\sigma}$ is the $3$-vector of Pauli matrices indexed by $1$, $2$ and $3$ and the dot also denotes the Euclidean inner product of two $3$-vectors. Note that in the previous string of equalities, we have made use of the algebra of Pauli matrices. For example,
\[ \sigma_i^2 = \mathbf{1} \quad \text{for $i = 1, \ldots, 3$} \] 
and
\[ \sigma_1 \sigma_2 = -\sigma_2 \sigma_1 = i \sigma_3 \]
and cyclically permuted versions of this formula over $\{1, 2, 3\}$.
But
\[ \operatorname{tr}(\mathbf{x}_i.\vec{\sigma}) = 0, \quad \text{for $i = 1, 2$,}\]
so we obtain
\begin{equation} |\langle \mathbf{w}_1 \, \mathbf{w}_2 \rangle|^2 = \frac{1}{2} \left( 1 + (\mathbf{x}_1, \mathbf{x}_2) \right). \label{2-cycle} \end{equation}
\begin{definition} Given $\mathbf{w}_i$ and $\mathbf{x}_i$ as above (for $i = 1, 2$), we define
\[ \rho_{12} = \frac{1}{2} \left( 1 + (\mathbf{x}_1, \mathbf{x}_2) \right).\]
\end{definition}
It follows from the Cauchy-Schwarz inequality that $0 \leq \rho_{12} \leq 1$.

We can rephrase what we have proved, as follows.
\begin{equation} |\langle \mathbf{w}_1 \, \mathbf{w}_2 \rangle|^2 = \rho_{12}. \label{2-cycle-a} \end{equation}

Let $\mathbf{w}_i \in \mathbb{C}^2$ for $i = 1, \ldots, 3$. We have
\begin{align*} 
& \langle \mathbf{w}_1 ,\, \mathbf{w}_2 \rangle \langle \mathbf{w}_2 ,\, \mathbf{w}_3 \rangle \langle \mathbf{w}_3 ,\, \mathbf{w}_1 \rangle \\
= & \, \mathbf{w}_2^* \mathbf{w}_1 \mathbf{w}_1^* \mathbf{w}_3 \mathbf{w}_3^* \mathbf{w}_2 \\
= & \, \operatorname{tr}\left( \mathbf{w}_1 \mathbf{w}_1^* \mathbf{w}_3 \mathbf{w}_3^* \mathbf{w}_2 \mathbf{w}_2^* \right) \\
= & \, \frac{1}{8} \operatorname{tr}\left[(\mathbf{1} - \mathbf{x}_1.\vec{\sigma}) (\mathbf{1} - \mathbf{x}_3.\vec{\sigma}) (\mathbf{1} - \mathbf{x}_2.\vec{\sigma})\right] \\
= & \, \frac{1}{8} \operatorname{tr}\left[(1 + (\mathbf{x}_1, \mathbf{x}_2) + (\mathbf{x}_1, \mathbf{x}_3) + (\mathbf{x}_2, \mathbf{x}_3)) \mathbf{1} - (\mathbf{x}_1.\vec{\sigma}) (\mathbf{x}_3.\vec{\sigma}) (\mathbf{x}_2.\vec{\sigma}) \right], \\
\end{align*}
where we have used, in the last line, the fact that $\operatorname{tr}(\mathbf{x}_i.\vec{\sigma}) = 0$ for $i = 1, \ldots, 3$. But
\[ (\mathbf{x}_3.\vec{\sigma}) (\mathbf{x}_2.\vec{\sigma}) = (\mathbf{x}_2, \mathbf{x}_3) \mathbf{1} - i (\mathbf{x}_2 \times \mathbf{x}_3).\vec{\sigma}, \]
so that
\begin{align*} 
& \langle \mathbf{w}_1 ,\, \mathbf{w}_2 \rangle \langle \mathbf{w}_2 ,\, \mathbf{w}_3 \rangle \langle \mathbf{w}_3 ,\, \mathbf{w}_1 \rangle \\
= & \, \frac{1}{8} \operatorname{tr}\left[(1 + (\mathbf{x}_1, \mathbf{x}_2) + (\mathbf{x}_1, \mathbf{x}_3) + (\mathbf{x}_2, \mathbf{x}_3) + i \det(\mathbf{x}_1, \mathbf{x}_2, \mathbf{x}_3)) \mathbf{1}\right] \\
= & \, \frac{1}{4} \left(1 + (\mathbf{x}_1, \mathbf{x}_2) + (\mathbf{x}_1, \mathbf{x}_3) + (\mathbf{x}_2, \mathbf{x}_3) + i \det(\mathbf{x}_1, \mathbf{x}_2, \mathbf{x}_3)\right),
\end{align*}
where $\det(\mathbf{x}_1, \mathbf{x}_2, \mathbf{x}_3)$ is the determinant of the real $3 \times 3$ matrix having $\mathbf{x}_1$, $\mathbf{x}_2$ and $\mathbf{x}_3$ as its $3$ columns, in that order.
Using the $\rho$ notation, we have proved
\begin{equation} \begin{split} & \langle \mathbf{w}_1 ,\, \mathbf{w}_2 \rangle \langle \mathbf{w}_2 ,\, \mathbf{w}_3 \rangle \langle \mathbf{w}_3 ,\, \mathbf{w}_1 \rangle \\
= & \, \frac{1}{2} \left(-1 + \rho_{12} + \rho_{13} + \rho_{23}\right) + \frac{i}{4} \det(\mathbf{x}_1, \mathbf{x}_2, \mathbf{x}_3).
\end{split} \label{3-cycle}
\end{equation}

\section{The Atiyah problem on configurations}

Let $C_n(\mathbb{R}^3)$ be the configuration space of $n$ distinct points in $\mathbb{R}^3$. Given $\mathbf{x} = (\mathbf{x}_1, \ldots, \mathbf{x}_n) \in C_n(\mathbb{R}^3)$, for each pair of indices $i$, $j$ with $1 \leq i, j \leq n$ and $i \neq j$, we let
\[ \nu_{ij} = \frac{\mathbf{x}_j - \mathbf{x}_i}{\lVert \mathbf{x}_j - \mathbf{x}_i \rVert} \in S^2. \]
We choose for each such pair of indices $i$, $j$ a Hopf lift $\mathbf{w}_{ij} \in S^3 \subset \mathbb{C}^2$ of $\nu_{ij}$, by which we mean that $h(\mathbf{w}_{ij}) = \nu_{ij}$. Later on, we will be using the modified Hopf map rather than the Hopf map, but for the time being, we will describe the problem first using the usual Hopf map.
Using $\omega$ defined in \eqref{omega}, we identify each $\mathbf{w}_{ij}$ with a corresponding homogeneous linear form $p_{ij}(\mathbf{w})$ on $\mathbb{C}^2$ via
\[ p_{ij}(\mathbf{w}) = \omega(\mathbf{w}_{ij}, \mathbf{w}), \]
where $\mathbf{w} = (u, v)^T$ are the coordinates of a variable point in $\mathbb{C}^2$.
Given $i$, with $1 \leq i \leq n$, we form
\[ p_i(\mathbf{w}) = \prod_{j \neq i} p_{ij}(\mathbf{w}), \]
which is a complex homogeneous polynomial of degree $n - 1$ in two complex variables ($u$ and $v$). Note that each $p_i(\mathbf{w})$ is well defined only up to a phase factor.
\begin{conjecture1} Given $n$, the polynomials $p_i(\mathbf{w})$, for $i = 1, \ldots, n$, associated to any given configuration $\mathbf{x} \in C_n(\mathbb{R}^3)$, are linearly independent (over $\mathbb{C}$).\end{conjecture1}
Conjecture 1 is actually a theorem by M. F. Atiyah for $n = 3$ (cf. \cite{Ati-2000}, \cite{Ati-2001}) and was proved by Eastwood and Norbury for $n = 4$ in \cite{EasNor2001}, though it is open for the general $n > 4$ case. The case $n = 4$ is the first ``hard'' case though, as discussed in the introduction.

\section{The Atiyah determinant} \label{At-det}

Building on the description in the previous section, we may consider ($u$, $v$) as a basis of $(\mathbb{C}^2)^*$ (the complex dual space of $\mathbb{C}^2$). We may similarly consider
\[ (u^{n-1}, u^{n-2}v, \ldots, u^iv^{n-1-i}, \ldots, v^{n-1}) \]
as a basis of $\operatorname{Sym}^{n-1}(\mathbb{C}^2)^*$, i.e. of the $n-1$-st symmetric tensor power of $(\mathbb{C}^2)^*$. With respect to these two bases respectively, we may think of $p_{ij}(\mathbf{w})$, respectively $p_i(\mathbf{w})$, as a vector in $\mathbb{C}^2$, respectively $\mathbb{C}^n$. So it does make sense then to talk about $\det(p_{ij}(\mathbf{w}), p_{ji}(\mathbf{w}))$, which is the determinant of the complex $2 \times 2$ matrix whose column vectors are $p_{ij}(\mathbf{w})$ and $p_{ji}(\mathbf{w})$ in that order.

\begin{definition} Atiyah's normalized determinant function $D: C_n(\mathbb{R}^3) \to \mathbb{C}$ is defined by
\[ D(\mathbf{x}) = \frac{\det(p_1(\mathbf{w}), \ldots, p_n(\mathbf{w}))}{\prod_{1 \leq i < j \leq n} \det(p_{ij}(\mathbf{w}), p_{ji}(\mathbf{w}))}. \]
\end{definition}

It can be checked that $D$ is well defined. Indeed, both the numerator and denominator are homogeneous in the $p_{ij}(\mathbf{w})$, where $1 \leq i, j \leq n$ and $i \neq j$, of degree $1$ in each single $p_{ij}(\mathbf{w})$, so that a choice of different phase factors for the $p_{ij}(\mathbf{w})$ would lead to the same ratio and thus to a well-defined value for $D(\mathbf{x})$.

As a remark, one may define $D(\mathbf{x})$ alternatively as follows and this is the definition we will adopt in this article. At the step where one chooses Hopf lifts, one may assume that once a choice of Hopf lift $\mathbf{w}_{ij}$ of $\nu_{ij}$ was made, where $1 \leq i < j \leq n$, then $\mathbf{w}_{ji}$ is then taken to be as follows.
\[ \begin{pmatrix} u_{ji} \\ v_{ji} \end{pmatrix} = 
\begin{pmatrix*}[r]
	-\bar{v}_{ij} \\ \bar{u}_{ij}
 \end{pmatrix*}. \]
It can be easily checked that $\mathbf{w}_{ji}$ is then a Hopf lift of $\nu_{ji} = -\nu_{ij}$. The experts will recognize that the map sending $\mathbf{w}_{ij}$ to $\mathbf{w}_{ji}$ above is a quaternionic structure on $\mathbb{C}^2$ and is what the antipodal map in $\mathbb{R}^3$ (the so-called parity transformation in physics) corresponds to on the $2$-spinor level.
That the $2$ definitions of $D$ are equivalent is left as an exercise to the reader (it is partly based on the observation that if we follow the above prescription, then $\det(p_{ij}(\mathbf{w}), p_{ji}(\mathbf{w})) = 1$).

\begin{conjecture2} Given $n$, for any given configuration $\mathbf{x} \in C_n(\mathbb{R}^3)$, $|D(\mathbf{x})| \geq 1$.\end{conjecture2}
Conjecture 2 is actually a theorem by M. F. Atiyah for $n = 3$ and was proved for $n = 4$ by Bou Khuzam and Johnson in \cite{KhuJoh2014} (who also proved the stronger Conjecture 3). At around the same time, Dragutin Svrtan gave a talk at the 73-rd ``Seminaire Lotharingien de Combinatoire'' presenting his arguments for the Atiyah-Sutcliffe conjectures 1-3 for the $n = 4$ case, though they do not seem to be published (to the best of the author's knowledge). Conjecture 2 is, at the time of writing, open for the general $n > 4$ case.

As an important note, we will diverge from the standard description of the Atiyah problem on configurations in the following points:
\begin{enumerate}
\item We make use of the modified Hopf map $\tilde{h}$ rather than $h$, so that we have, for example,
\[ \tilde{h}(\mathbf{w}_{ij}) = \nu_{ij} \quad \text{($1 \leq i, j \leq n$, $i \neq j$). }\]
\item We do not ``dualize'' each $\mathbf{w}_{ij}$ using the complex symplectic form $\omega$. We also replace polynomial multiplication with the symmetric tensor product (over $\mathbb{C}$). Thus, for example, instead of considering $p_{ij}(\mathbf{w})$, we consider instead $p_{ij} = \mathbf{w}_{ij}$ (for $1 \leq i, j \leq n$ and $i \neq j$) and we define, for $i = 1, \ldots, n$, the following:
\[ p_i = \operatorname{Sym}_{j \neq i} p_{ij} \in \operatorname{Sym}^{n-1}(\mathbb{C}^2), \]
where $\operatorname{Sym}$ denotes the symmetric tensor product (over $\mathbb{C}$). The vector $p_i$ will play the role of $p_i(\mathbf{w})$.
\end{enumerate}

One may mimic the definition of $D$ above, making sure to use the above modifications. Using the modified Hopf map instead of the Hopf map has the effect of multiplying both numerator and denominator by the same sign factor, namely $(-1)^{\binom{n}{2}}$, so that the value of $D(\mathbf{x})$ remains the same. Similarly, using or not the matrix $\omega$ to dualize $\mathbf{w}_{ij}$ has no effect on the numerator, nor the denominator, in the definition of $D(\mathbf{x})$, since $\omega$ has determinant equal to $1$.

\section{The Gram matrix}

If $\mathbb{C}^2$ is equipped with the standard hermitian inner product \eqref{herm}, then this induces an hermitian inner product on $\operatorname{Sym}^m(\mathbb{C}^2)$, the $m$-th symmetric tensor power of $\mathbb{C}^2$, where $m$ is any given positive integer. More precisely, if $p = \mathbf{w}_1 \odot \cdots \odot \mathbf{w}_m$ and $p' = \mathbf{w}'_1 \odot \cdots \odot \mathbf{w}'_m$, where $\odot$ denotes the symmetric tensor product (which some people may call the symmetric Kronecker product), we then define
\[ \langle p, \, p' \rangle = \sum_{\sigma \in S_m} \langle \mathbf{w}_i, \, \mathbf{w}'_{\sigma(i)} \rangle. \]
Note that we do not include a normalization factor, though many authors may choose to include one.

We fix an integer $n \geq 2$. Given $\mathbf{x} \in C_n(\mathbb{R}^3)$, we let $p_1, \ldots, p_n$ be the corresponding vectors in $\operatorname{Sym}^{n-1}(\mathbb{C}^2)$, each defined up to a phase factor. We form the Gram matrix $H_n(\mathbf{x})$ (which we may sometimes simply write as $H_n$, if $\mathbf{x}$ is understood) of these $n$ vectors, namely
\[ H_n = (\langle p_i,\, p_j \rangle) , \quad \text{$1 \leq i,j \leq n$},\]
where $\langle -,\, - \rangle$ is the hermitian inner product induced on $\operatorname{Sym}^{n-1}(\mathbb{C}^2)$ by the standard hermitian inner product on $\mathbb{C}^2$, as described above.

We note however that, while $H_n$ itself is only defined up to conjugation by an $n \times n$ diagonal unitary matrix due to the phase ambiguity, its eigenvalues (with multiplicity) are well defined and so is its determinant. As a matter of fact, if one keeps track carefully of the norms of the vectors
\[ e_1^{\odot k} \odot e_2^{\odot n-1-k} \quad \text{($0 \leq k \leq n-1$)}\]
(where $e_1 = (1, 0)^T$ and $e_2 = (0, 1)^T$), which form a basis of $\operatorname{Sym}^{n-1}(\mathbb{C}^2)$ used in the definition of $D(\mathbf{x})$, one may show that
\begin{equation} \det(H_n(\mathbf{x})) = \left(\prod_{k=0}^{n-1} \lVert e_1^{\odot k} \odot e_2^{\odot n-1-k} \rVert^2 \right) |D|^2 = c_n |D|^2. \label{det_gram1} \end{equation}
where
\begin{equation} c_n = \prod_{k=0}^{n-1} (k!)^2. \label{det_gram2} \end{equation}

Being a Gram matrix, $H_n(\mathbf{x})$ is therefore always positive semidefinite. Moreover, $H_n(\mathbf{x})$ is positive definite iff its determinant is positive (since we already know that all its eigenvalues are real and nonnegative) iff the $n$ vectors $p_i$, $1 \leq i \leq n$, are linearly independent over $\mathbb{C}$, i.e. iff conjecture 1 holds.

Moreover, one may easily see using \eqref{det_gram1} and \eqref{det_gram2}, that conjecture 2 is equivalent to
\[ \det{H_n} \geq c_n. \]

The following is a reformulation of known results.
\begin{theorem} If $n \leq 4$, then the Gram matrix $H_n(\mathbf{x})$ is positive definite for any $\mathbf{x} \in C_n(\mathbb{R}^3)$. \end{theorem}

We stress that, as a statement, this is only a reformulation of the fact that Atiyah's conjecture 1 is known to be true for $n \leq 4$ (proved by M. F. Atiyah for $n = 3$, cf \cite{Ati-2000}, \cite{Ati-2001}, and by Eastwood and Norbury for $n = 4$ in \cite{EasNor2001}). Our contribution is in providing a new proof for the $n = 4$ case (which is the first ``hard'' case) that does not require the full expansion of the $4 \times 4$ Atiyah determinant, nor the use of computer algebra software such as Maple, unlike the ``tour de force'' kind of proof found in \cite{EasNor2001}. The author hopes that this work will complement Eastwood and Norbury's work and may allow us to proceed further and possibly tackle the $n > 4$ case (ideally, at least).

\section{Proof of the $n = 3$ case}

In this section, we tackle the case of $n = 3$ points in $\mathbf{R}^3$. Given $\mathbf{x} \in C_3(\mathbb{R}^3)$, the corresponding Gram matrix $H_3(\mathbf{x})$ is given by
\begin{equation} H_3(\mathbf{x}) = 
\begin{pmatrix} 1 + |\herm{12}{13}|^2 & \herm{12}{23} \,\herm{13}{21} & \herm{12}{31} \,\herm{13}{32} \\
\herm{23}{12} \,\herm{21}{13}& 1 + |\herm{21}{23}|^2 & \herm{21}{32} \,\herm{23}{31} \\
\herm{31}{12} \,\herm{32}{13} & \herm{32}{21} \,\herm{31}{23} & 1 + |\herm{31}{32}|^2 \end{pmatrix}, \label{H3} \end{equation}
where $h_{ij, kl} = \langle \mathbf{w}_{ij},\, \mathbf{w}_{kl} \rangle$.

Let
\begin{align*} \mu_1 &= |\herm{12}{13}|^2 = 1 - |\herm{31}{12}|^2 \\
\mu_2 &= |\herm{21}{23}|^2 = 1 - |\herm{12}{23}|^2 \\
\mu_3 &= |\herm{31}{32}|^2 = 1 - |\herm{23}{31}|^2.
\end{align*}
To understand the first line in the previous set of formulas, note that $|\herm{12}{13}|^2 + |\herm{31}{12}|^2 = 1$ since $\mathbf{w}_{12}$ has unit norm and $(\mathbf{w}_{13}, \mathbf{w}_{31})$ is a unitary basis of $\mathbb{C}^2$. The other two lines are similar.

We now expand $\det(H_3(\mathbf{x}))$, obtaining
\[ \begin{split} & \det(H_3(\mathbf{x})) \\ 
	= \, &(1 + \mu_1)(1 + \mu_2)(1 + \mu_3) + 2(1 - \mu_1)(1 - \mu_2)(1 - \mu_3) - \cdots \\
 & \cdots -(1 + \mu_1)(1 - \mu_2)(1 - \mu_3) - (1 - \mu_1)(1 + \mu_2)(1 - \mu_3) - \cdots \\
 & \cdots -(1 - \mu_1)(1 - \mu_2)(1 + \mu_3) \end{split}
\]
Expanding and simplifying, we obtain
\begin{equation} \det(H_3(\mathbf{x})) = 4(\mu_1 \mu_2 + \mu_1 \mu_3 + \mu_2 \mu_3) - 4 \mu_1 \mu_2 \mu_3. \label{det-H3} \end{equation}
We consider the Gram matrix $G$ of the vectors $\nu_{23}$, $\nu_{31}$ and $\nu_{12}$ in $\mathbb{R}^3$, whose entries are the pairwise Euclidean inner products of these $3$ vectors. Using \eqref{2-cycle}, we find that
\[ G = \begin{pmatrix} 1 & 1 - 2 \mu_3 & 1 - 2 \mu_2 \\
	1 - 2 \mu_3 & 1 & 1 - 2 \mu_1 \\
	1 - 2 \mu_2 & 1 - 2 \mu_1 & 1 \end{pmatrix}. \]
But the $3$ vectors $\nu_{23}$, $\nu_{31}$ and $\nu_{12}$ are coplanar, so that $G$ has vanishing determinant. Hence
\[ \begin{split} 0 = \det(G) = \, 1 + 2(1-2\mu_1)(1-2\mu_2)(1-2\mu_3) &- (1-2\mu_1)^2 - \cdots \\
\cdots &-(1-2\mu_2)^2 -(1-2\mu_3)^2, \end{split} \]
which gives, after simplifying
\[ 0 = \det(G) = 8(\mu_1\mu_2 + \mu_1\mu_3 + \mu_2\mu_3) - 4(\mu_1^2 + \mu_2^2 + \mu_3^2) - 16\mu_1 \mu_2 \mu_3, \]
from which we obtain
\begin{equation}-4 \mu_1 \mu_2 \mu_3 = \mu_1^2 + \mu_2^2 + \mu_3^2 - 2(\mu_1 \mu_2 + \mu_1 \mu_3 + \mu_2 \mu_3). \label{identity} \end{equation}
Substituting the previous formula into \eqref{det-H3}, we finally obtain
\begin{equation} \det(H_3(\mathbf{x})) = (\mu_1 + \mu_2 + \mu_3)^2. \end{equation}
Note that the RHS of the previous formula is nonnegative, and cannot vanish. Indeed, if the RHS vanished, then each of $\mu_1$, $\mu_2$ and $\mu_3$ must vanish, which would imply that the triangle with vertices $\mathbf{x}_1$, $\mathbf{x}_2$ and $\mathbf{x}_3$ has all $3$ interior angles equal to $\pi$, which is clearly impossible (since the sum of the interior angles of a Euclidean triangle must be $\pi$).

Hence, we have proved that $\det(H_3(\mathbf{x})) > 0$. Thus $H_3(\mathbf{x})$ is an hermitian positive semidefinite matrix with positive determinant, from which we conclude that $H_3(\mathbf{x})$ is positive definite, for any $\mathbf{x} \in C_3(\mathbb{R}^3)$.

As a note, using \eqref{identification} and the fact that $D$ is real and positive if $n = 3$ (cf. \cite{Ati-Sut-2002}), we deduce that
\[ D = \frac{\mu_1 + \mu_2 + \mu_3}{2}, \]
(which is equivalent to formula (3.16) in \cite{Ati-Sut-2002}) if $n = 3$. Moreover, it is known that in this case, $D$ is minimized at a collinear configuration at which $D = 1$ and is maximized at an equilateral triangle at which $D = 9/8$ (also cf. \cite{Ati-2001},  \cite{Ati-Sut-2002}). We thus have, if $n = 3$, the following.
\begin{equation} 2 \leq \mu_1 + \mu_2 + \mu_3 \leq \frac{9}{4}, \label{bounds} \end{equation}
from which it follows that
\[ \det(H_3(\mathbf{x})) \geq 4. \]

Before leaving this section, we will prove the following lemma.

\begin{lemma} Given any $\mathbf{x} \in C_3(\mathbb{R}^3)$, $H_3(\mathbf{x}) - \mathbf{1}$ (where $\mathbf{1}$ here denotes the $3 \times 3$ identity matrix) is (hermitian) positive semidefinite.
\end{lemma}

\begin{proof}
Consider
\begin{equation} H_3(\mathbf{x}) - \mathbf{1} = 
\begin{pmatrix} |\herm{12}{13}|^2 & \herm{12}{23} \,\herm{13}{21} & \herm{12}{31} \,\herm{13}{32} \\
\herm{23}{12} \,\herm{21}{13}& |\herm{21}{23}|^2 & \herm{21}{32} \,\herm{23}{31} \\
\herm{31}{12} \,\herm{32}{13} & \herm{32}{21} \,\herm{31}{23} & |\herm{31}{32}|^2 \end{pmatrix}. \label{H3-minus-1} \end{equation}

It is clear that all the entries on the diagonal are nonnegative. Let us consider the leading principal $2 \times 2$ minor $m_{12}$ of the above matrix, given by
\begin{align*} m_{12} &= |\herm{12}{13}|^2 \,|\herm{21}{23}|^2 - |\herm{12}{23}|^2 \,|\herm{13}{21}|^2. \\
&= \mu_1 \mu_2 - (1 - \mu_2) (1 - \mu_1) \\
&= \mu_1 + \mu_2 - 1 \\
&\geq 1 - \mu_3 \\
&\geq 0,
\end{align*}
where we have used \eqref{bounds}. We can similarly prove that the other principal $2 \times 2$ minors of $H_3(\mathbf{x}) - \mathbf{1}$, namely $m_{13}$ and $m_{23}$, are also nonnegative. It remains only to show that the determinant of $H_3(\mathbf{x}) - \mathbf{1}$ is also nonnegative. Expanding, we have
\[ \begin{split}
& \det(H_3(\mathbf{x}) - \mathbf{1}) \\
= \, & \mu_1 \mu_2 \mu_3 + 2(1 - \mu_2)(1 - \mu_1)(1 - \mu_3) - \mu_1(1 - \mu_2)(1 - \mu_3) - \cdots \\
& \cdots - \mu_2(1 - \mu_1)(1 - \mu_3) - \mu_3(1 - \mu_1)(1 - \mu_2),
\end{split}
\]
where we have used that
\begin{align*} \herm{21}{32} &= \overline{\herm{12}{23}} \\
\herm{13}{21} &= - \overline{\herm{31}{12}} \\
\herm{32}{13} &= - \overline{\herm{23}{31}}.
\end{align*}
To understand the minus signs in the last $2$ equations, note that we are using the convention that if $1 \leq i < j \leq n$, then $w_{ji}$ is the quaternionic structure of $\mathbb{C}^2$ applied to $w_{ij}$ (see the paragraph in section \ref{At-det} preceding the statement of Conjecture 2). In particular, this implies that $w_{ij}$ is minus the quaternionic structure applied to $w_{ji}$. We note also that this quaternionic structure is anti-unitary with respect to the standard hermitian inner product on $\mathbb{C}^2$, which explains the appearance of complex conjugation in the above formulas.

Expanding and simplifying, we obtain
\[ \begin{split} & \det(H_3(\mathbf{x}) - \mathbf{1}) \\
= & \,2 - 3(\mu_1 + \mu_2 + \mu_3) + 4(\mu_1\mu_2 + \mu_1 \mu_3 + \mu_2 \mu_3) - 4 \mu_1 \mu_2 \mu_3.
\end{split}
\]
We now invoke \eqref{identity}, thus obtaining
\[ \det(H_3(\mathbf{x}) - \mathbf{1}) = 2 - 3(\mu_1 + \mu_2 + \mu_3) + (\mu_1 + \mu_2 + \mu_3)^2.\]
Factoring out the previous equation, we get
\begin{equation} \det(H_3(\mathbf{x}) - \mathbf{1}) = (\mu_1 + \mu_2 + \mu_3 - 2) (\mu_1 + \mu_2 + \mu_3 - 1),
\end{equation}
which is nonnegative, since $\mu_1 + \mu_2 + \mu_3 \geq 2$, from \eqref{bounds}. This finishes the proof of the lemma.
\end{proof}

\section{Proof of the $n = 4$ case}

Just as in the previous section, we make use of the notation
\[ \herm{ij}{kl} = \langle \mathbf{w}_{ij},\, \mathbf{w}_{kl} \rangle \]
($1 \leq i, j, k, l \leq 4$).

As a warm-up, we first expand
\[ \begin{split} & \langle p_1,\, p_1 \rangle \\
= \, & 1 + \herm{12}{13}\,\herm{13}{14}\,\herm{14}{12} +
\herm{13}{12}\,\herm{14}{13}\,\herm{12}{14}
- |\herm{12}{13}|^2 - |\herm{12}{14}|^2 - |\herm{13}{14}|^2.
\end{split} 
\]

It is well known that this hermitian inner product can be written as the permanent of a $3 \times 3$ matrix. More specifically, consider the Gram matrix $T_{11}$ of $(\mathbf{w}_{12}, \mathbf{w}_{13}, \mathbf{w}_{14})$, defined by
\[T_{11} = \begin{pmatrix} 1 & \herm{12}{13} & \herm{12}{14} \\
	\herm{13}{12} & 1 & \herm{13}{14} \\
	\herm{14}{12} & \herm{14}{13} & 1 \end{pmatrix}. \]
Then $\langle p_1,\, p_1 \rangle = \operatorname{perm}(T_{11})$. But $\mathbf{w}_{12}$, $\mathbf{w}_{13}$ and $\mathbf{w}_{14}$ are $3$ vectors in $\mathbb{C}^2$, which must thus be linearly dependent (over $\mathbb{C}$). Hence their Gram matrix $T_{11}$ is singular. We therefore have
\[ \det(T_{11}) = 0.\]
Using the previous formula, we deduce the following.
\begin{equation}\langle p_1,\, p_1 \rangle = 2(|\herm{12}{13}|^2 + |\herm{12}{14}|^2 + |\herm{13}{14}|^2). \label{p1-p1}\end{equation}

Using a similar approach, we compute $\langle p_1,\, p_2 \rangle$, which is the permanent of the following $3 \times 3$ matrix
\[T_{12} = \begin{pmatrix} 0 & \herm{12}{23} & \herm{12}{24} \\
	\herm{13}{21} & \herm{13}{23} & \herm{13}{24} \\
	\herm{14}{21} & \herm{14}{23} & \herm{14}{24} \end{pmatrix}. \]
But $\mathbf{w}_{12}$, $\mathbf{w}_{13}$ and $\mathbf{w}_{14}$ are $3$ vectors in $\mathbb{C}^2$, which must thus be linearly dependent (over $\mathbb{C}$), so that $T_{11}$, which is the ``mixed'' Gram matrix between $(\mathbf{w}_{12}, \mathbf{w}_{13}, \mathbf{w}_{14})$ and $(\mathbf{w}_{21}, \mathbf{w}_{23}, \mathbf{w}_{24})$, must be singular. Hence
\[ \det(T_{12}) = 0. \]
Using the previous formula, we obtain that
\begin{equation} \langle p_1,\, p_2 \rangle = \operatorname{per}(T_{12}) = 2(\herm{12}{23}\, \herm{13}{21}\, \herm{14}{24} + \herm{12}{24}\, \herm{13}{23}\, \herm{14}{21}).\label{p1-p2}\end{equation}
There are formulas similar to \eqref{p1-p1} and \eqref{p1-p2} for any $\langle p_i,\, p_j \rangle$.
Let $H_{123}(\mathbf{x})$ be the Gram matrix of the configuration $\mathbf{x}_1$, $\mathbf{x}_2$ and $\mathbf{x}_3$ (i.e. we delete $\mathbf{x}_4$ from the original configuration). Hence $H_{123}(\mathbf{x}) - \mathbf{1}$, where $\mathbf{1}$ is the $3 \times 3$ identity matrix, is given by the RHS of \eqref{H3-minus-1}.

If $\mathbf{w}_i \in \mathbb{C}^2$, for $i = 1, \ldots, 3$, we define their Gram matrix to be
\[ H(\mathbf{w}_1, \mathbf{w}_2, \mathbf{w}_3) = (\langle \mathbf{w}_i,\, \mathbf{w}_j \rangle) \]
($1 \leq i,j \leq 3$), which is a singular hermitian positive semidefinite matrix.

We now define
\[ A_4(\mathbf{x}) = 
\begin{pmatrix} \tilde{A}_4 & \mathbf{0} \\
\mathbf{0}^T & 0 \end{pmatrix}, \]
where
\[ \tilde{A}_4 = 2(H_{123} - \mathbf{1}) * H(\mathbf{w}_{14}, \mathbf{w}_{24}, \mathbf{w}_{34}),\]
with $*$ denoting the Hadamard product of matrices and where we have omitted the dependence on $\mathbf{x}$ from our notation, for brevity. It is known that the Hadamard product of two hermitian positive semidefinite matrices is also hermitian positive semidefinite, from which we deduce that $\tilde{A}_4$, and thus also $A_4$, is hermitian positive semidefinite.

Similarly, if $1 \leq i \leq 3$, we can define $\tilde{A}_i$ in a similar fashion. More precisely, we define
\begin{align*}
\tilde{A}_1 &= 2(H_{234} - \mathbf{1}) * H(\mathbf{w}_{21}, \mathbf{w}_{31}, \mathbf{w}_{41}) \\
\tilde{A}_2 &= 2(H_{134} - \mathbf{1}) * H(\mathbf{w}_{12}, \mathbf{w}_{32}, \mathbf{w}_{42}) \\
\tilde{A}_3 &= 2(H_{124} - \mathbf{1}) * H(\mathbf{w}_{13}, \mathbf{w}_{23}, \mathbf{w}_{43})
\end{align*}
For example $H_{234}$ is the $3 \times 3$ Gram matrix of the configuration $\mathbf{x}$ from which we have deleted the point $\mathbf{x}_1$.

We let $A_i$ ($1 \leq i \leq 3$) be the $4 \times 4$ matrix having $\tilde{A}_i$ as its $3 \times 3$ principal submatrix with row and column indices taken from $\{1, 2, 3, 4\} \setminus \{i\}$ and having zeros everywhere else.

Just as for $\tilde{A}_4$ and $A_4$, it is clear that the $\tilde{A}_i$ and $A_i$ ($1 \leq i \leq 3$) are all hermitian positive semidefinite.

By examining carefully the formulas for the $\langle p_i,\, p_j \rangle$ ($1 \leq i,j \leq 4$), for which \eqref{p1-p1} and \eqref{p1-p2} form a representative sample, one deduces the following fundamental decomposition for $H_4(\mathbf{x})$,

\begin{equation} H_4(\mathbf{x}) = \sum_{i = 1}^4 A_i(\mathbf{x}). \label{decomposition}\end{equation}

We now focus our attention onto $\tilde{A}_4(\mathbf{x})$, as we would like to know for which $\mathbf{x} \in C_4(\mathbb{R}^3)$ it is positive definite, which amounts to checking for which $\mathbf{x} \in C_4(\mathbb{R}^3)$ the determinant of $\tilde{A}_4(\mathbf{x})$ is positive.

We introduce some notation:
\begin{align*}
&\mu_1 = |\herm{12}{13}|^2, &\mu_2 = |\herm{21}{23}|^2, \qquad &\mu_3 = |\herm{31}{32}|^2, \\
&\rho_1 = |\herm{24}{34}|^2, &\rho_2 = |\herm{34}{14}|^2, \qquad &\rho_3 = |\herm{14}{24}|^2.
\end{align*}
We also let
\[ \tilde{\mu}_i = 1 - \mu_i \quad \text{and} \quad \tilde{\rho}_i = 1 - \rho_i \qquad \text{($1 \leq i \leq 3$)}. \]
Note that $\mu_i$, $\rho_i$, $\tilde{\mu}_i$ and $\tilde{\rho}_i$ ($1 \leq i \leq 3$) all lie in $[0, 1]$.

Using \eqref{3-cycle}, we get that
\[ \Re(\herm{14}{24}\,\herm{24}{34}\,\herm{34}{14}) = \frac{1}{2}(-1 + T),\]
where $T = \rho_1 + \rho_2 + \rho_3$.

Expanding, we get
\[\begin{split} & \det(\tilde{A}_4/2) \\
= & \,\mu_1 \mu_2 \mu_3 + \tilde{\mu}_1 \tilde{\mu}_2 \tilde{\mu}_3(-1 + T) - \tilde{\mu}_1 \tilde{\mu}_2 \mu_3 \rho_3 - \tilde{\mu}_1 \mu_2 \tilde{\mu}_3 \rho_2 - \mu_1 \tilde{\mu}_2 \tilde{\mu}_3 \rho_1.
\end{split}
\]
Letting $\widetilde{T} = 3 - T \geq 0$, we can rewrite the previous equation as follows.
\[\begin{split} & \det(\tilde{A}_4/2) \\
= & \,\mu_1 \mu_2 \mu_3 + 2\,\tilde{\mu}_1 \tilde{\mu}_2 \tilde{\mu}_3 - \tilde{\mu}_1 \tilde{\mu}_2 \mu_3 - \tilde{\mu}_1 \mu_2 \tilde{\mu}_3 - \mu_1 \tilde{\mu}_2 \tilde{\mu}_3 - \tilde{\mu}_1 \tilde{\mu}_2 \tilde{\mu}_3 \widetilde{T} + \cdots \\
& \qquad \qquad \cdots + \tilde{\mu}_1 \tilde{\mu}_2 \mu_3 \tilde{\rho}_3 + \tilde{\mu}_1 \mu_2 \tilde{\mu}_3 \tilde{\rho}_2 + \mu_1 \tilde{\mu}_2 \tilde{\mu}_3 \tilde{\rho}_1 \\
= & \, \det(H_{123}(\mathbf{x}) - \mathbf{1}) - \tilde{\mu}_1 \tilde{\mu}_2 \tilde{\mu}_3\widetilde{T} + \tilde{\mu}_1 \tilde{\mu}_2 \mu_3 \tilde{\rho}_3 + \tilde{\mu}_1 \mu_2 \tilde{\mu}_3 \tilde{\rho}_2 + \mu_1 \tilde{\mu}_2 \tilde{\mu}_3 \tilde{\rho}_1
\end{split}
\]
Let $S = \mu_1 + \mu_2 + \mu_3$. It can be shown, after expanding and using \eqref{identity}, that
\begin{equation}
4 \tilde{\mu}_1 \tilde{\mu}_2 \tilde{\mu}_3 = (S - 2)^2.
\label{identity2}
\end{equation}
Using \eqref{H3-minus-1} and \eqref{identity2}, we obtain
\[\begin{split} & \frac{1}{2}\det(\tilde{A}_4) \\
= & \, 4(S - 2)(S - 1) - (S - 2)^2 \widetilde{T} + 4(\tilde{\mu}_1 \tilde{\mu}_2 \mu_3 \tilde{\rho}_3 + \tilde{\mu}_1 \mu_2 \tilde{\mu}_3 \tilde{\rho}_2 + \mu_1 \tilde{\mu}_2 \tilde{\mu}_3 \tilde{\rho}_1).
\end{split}
\]
But $\widetilde{T} = 3 - T$, so we obtain
\begin{equation} \frac{1}{2}\det(\tilde{A}_4) = (S - 2)(S + 2) + R, \label{main-eq} \end{equation}
where
\begin{equation} R = (S - 2)^2 T + 4(\tilde{\mu}_1 \tilde{\mu}_2 \mu_3 \tilde{\rho}_3 + \tilde{\mu}_1 \mu_2 \tilde{\mu}_3 \tilde{\rho}_2 + \mu_1 \tilde{\mu}_2 \tilde{\mu}_3 \tilde{\rho}_1) \geq 0.
\label{main-eq2} \end{equation}
But we have seen that, from the previous section (see \eqref{bounds}), that $S \geq 2$ and that equality is attained if $\mathbf{x}_1$, $\mathbf{x}_2$ and $\mathbf{x}_3$ are collinear. We also claim that equality is attained \emph{only if} $\mathbf{x}_1$, $\mathbf{x}_2$ and $\mathbf{x}_3$ are collinear. This can be seen as follows. If $\alpha$, $\beta$ and $\gamma$ are the interior angles of the triangle with vertices $\mathbf{x}_1$, $\mathbf{x}_2$ and $\mathbf{x}_3$, we then have
\[ S - 2 = \operatorname{cos}^2\left(\frac{\alpha}{2}\right) + \operatorname{cos}^2\left(\frac{\beta}{2}\right) + \operatorname{cos}^2\left(\frac{\gamma}{2}\right) - 2, \]
which can be manipulated using some trigonometric identities and shown to yield
\[ S - 2 = 2 \,\operatorname{sin}\left(\frac{\alpha}{2}\right)\, \operatorname{sin}\left(\frac{\beta}{2}\right)\, \operatorname{cos}\left(\frac{\alpha+\beta}{2}\right) \geq 0, \]
since $0 \leq \frac{\alpha}{2}, \frac{\beta}{2}, \frac{\gamma}{2} \leq \frac{\pi}{2}$ and $\alpha + \beta = \pi - \gamma$. We can also see that $S - 2$ vanishes iff one of the interior angles ($\alpha$, $\beta$ and $\gamma$) vanishes, i.e. iff $\mathbf{x}_1$, $\mathbf{x}_2$ and $\mathbf{x}_3$ are collinear.

Going back to \eqref{main-eq} and \eqref{main-eq2}, we easily see that, if $\mathbf{x}_1$, $\mathbf{x}_2$ and $\mathbf{x}_3$ are collinear, then $S = 2$ and $R = 0$, so that $\det(\tilde{A}_4) = 0$, while if they are not collinear, then $\det(\tilde{A}_4) > 0$.

In other words, $\tilde{A}_4$ is (hermitian) positive definite iff the points $\mathbf{x}_1$, $\mathbf{x}_2$ and $\mathbf{x}_3$ are \emph{not} collinear.

We now go back to our main argument. If the $4$ points $\mathbf{x}_1, \ldots, \mathbf{x}_4$ are collinear, then it can be verified directly that $H_4(\mathbf{x})$ is positive definite.

We then assume that $\mathbf{x}_1, \ldots, \mathbf{x}_4$ are not collinear. It is therefore possible to remove one of the four points and still get a non-collinear configuration, so we assume, WLOG, that $\mathbf{x}_1$, $\mathbf{x}_2$ and $\mathbf{x}_3$ are not collinear. There is also a $3$-subset of the configuration of four points $\mathbf{x}$ containing $\mathbf{x}_4$ which is also not collinear. We thus assume, WLOG, that $\mathbf{x}_2$, $\mathbf{x}_3$ and $\mathbf{x}_4$ are not collinear.

Let $v \in \mathbb{C}^4$ such that
\[ \langle H_4(\mathbf{x}) v,\, v \rangle = 0, \]
where $\langle -,\, - \rangle$ now denotes the standard hermitian inner product on $\mathbb{C}^4$. Using \eqref{decomposition}, we obtain
\[ 0 = \sum_{i=1}^4 \langle A_i(\mathbf{x}) v, \, v \rangle. \]
Since the $A_i$ ($1 \leq i \leq 4$) are hermitian positive semidefinite, we therefore have
\[ \langle A_i(\mathbf{x}) v, \, v \rangle = 0, \quad \text{for $i = 1, \dots, 4$}. \]
Since $\mathbf{x}_1$, $\mathbf{x}_2$ and $\mathbf{x}_3$ are not collinear (using our assumptions above), we thus have that $\tilde{A}_4$ is positive definite. Moreover, since we have, in addition, that
\[ \langle A_4(\mathbf{x}) v, \, v \rangle = 0, \]
we therefore deduce that the first $3$ components of $v$ are $0$. Similarly, since $\mathbf{x}_2$, $\mathbf{x}_3$ and $\mathbf{x}_4$ are also assumed to be non-collinear, we therefore obtain that the last $3$ components of $v$ are $0$. Hence $v$ vanishes. We have therefore finished the proof that $H_4(\mathbf{x})$ is hermitian positive definite for any $\mathbf{x} \in C_4(\mathbb{R}^3)$.

\section{Future work}

In a future work, the author would like to attempt to apply this method to the general $n > 4$ case and see how far he would get.

\section*{Acknowledgements}

The author would like to thank Vivecca for her love and support. He also wishes to thank Dennis Sullivan, Peter Olver and Niky Kamran for their patience, after having received many emails about this problem from him, and for their moral support. Many thanks also go to Stephen Drury for some interesting discussions about the permanent of a matrix and various related inequalities. The author would also like to thank Paul Cernea for listening to his ideas and asking relevant questions.


\begin{thebibliography}{1}
	
	\bibitem{Ati-2000}
	Atiyah, M. F., \emph{The geometry of classical particles}, Surveys in differential geometry, 1 - 15, Surv. Differ. Geom., VII, Int. Press, Somerville, MA, 2002.
	
	\bibitem{Ati-2001}
	Atiyah, M. F., \emph{Configurations of points}, Topological methods in the physical sciences (London, 2000). R. Soc. Lond. Philos. Trans. Ser. A Math. Phys. Eng. Sci. \textbf{359} (2001), no. 1784, 1375 - 1387.
	
	\bibitem{Ati-Sut-2002}
	Atiyah, M. F. and Sutcliffe, P. M., \emph{The geometry of point particles}, R. Soc. Lond. Proc. Ser. A Math. Phys. Eng. Sci. \textbf{458} (2002), no. 2021, 1089 - 1115.
	
	\bibitem{BR1997}
	Berry, M. V. and Robbins, J. M., \emph{Indistinguishability for quantum particles: spin, statistics and the geometric phase}, Proc. Roy. Soc. London Ser. A \textbf{453} (1997), 1771--1790.
	
	\bibitem{KhuJoh2014}
	Bou Khuzam, M. N. and Johnson, M. J., \emph{On the Conjectures Regarding the 4-Point Atiyah Determinant}, SIGMA, \textbf{10} (2014), 070, 9 pp.

	\bibitem{EasNor2001}
	Eastwood, M. G. and Norbury, P., \emph{A proof of Atiyah's conjecture on configurations of four points in Euclidean three-space}, Geom. Topol. 5 (2) 885 - 893, 2001.

	
\end{thebibliography}
\end{document}